\documentclass{svmult}
\usepackage{amsmath,amssymb}
\usepackage{enumerate}

\usepackage{mathptmx}       % selects Times Roman as basic font
\usepackage{helvet}         % selects Helvetica as sans-serif font
\usepackage{courier}        % selects Courier as typewriter font
\usepackage{type1cm}        % activate if the above 3 fonts are
\usepackage{makeidx}         % allows index generation
\usepackage{graphicx}        % standard LaTeX graphics tool
\usepackage{multicol}        % used for the two-column index
\usepackage[bottom]{footmisc}% places footnotes at page bottom

\makeindex             % used for the subject index
\begin{document}
\title*{Approximation and stability of solutions of SDEs driven by a symmetric $\alpha$ stable process with non-Lipschitz coefficients.}

\author{Hiroya Hashimoto}

\titlerunning{Approximation and stability of solutions of SDEs driven by a S$\alpha$S process.}

\authorrunning{H. Hashimoto}

\institute{H. Hashimoto \at Sanwa Kagaku Kenkyusho Co.,LTD. Nagoya, Japan \\ \email{hiro\_hashimoto@skk-net.com}}

\maketitle
\abstract{
Firstly, we investigate Euler-Maruyama approximation for solutions of stochastic differential equations (SDEs) driven by a symmetric $\alpha$ stable process under Komatsu condition for coefficients.
The approximation implies naturally the existence of strong solutions. 
Secondly, we study the stability of solutions under Komatsu condition, and also discuss it under Belfadli-Ouknine condition.}
\section{Introduction}\label{sec:1}
Euler-Maruyama approximation is a key tool in the theory of stochastic differential equations (SDEs) as well as Picard approximation is.
In this domain the theory on stability properties of solutions is considered as one of the cornerstones.
This article is devoted firstly to study Euler-Maruyama approximation in the pathwise sense, and secondly to investigate stability problems also in the pathwise sense.

We consider these problems in the case where SDEs with non-Lipschitz coefficients are driven by a symmetric $\alpha$ stable process ($1 < \alpha < 2$).
SDEs driven by a symmetric $\alpha$ stable process more generally those of pure jumps type arise naturally in connection with applications, for example, mathematical finance.

We briefly sketch some known results in this area.
In the case where the driving process is a Brownian motion ($\alpha = 2$), Euler-Maruyama approximation in the pathwise sense has been shown under non-Lipschitz condition for coefficients (Yamada \cite{Yamada}, see also Kaneko-Nakao \cite{K-N}).

Stability problems for solutions have been developed very well by \'{E}mery \cite{Emery} and Protter \cite{Protter}, in the framework of SDEs driven by semimartingales with Lipschitz coefficients.

Stability problems in law sense for martingale problem solutions of SDEs driven by jump processes have been discussed by many authors, for example, Kasahara-Yamada \cite{Ks-Y} and Janicki-Michna-Weron \cite{J-M-W}.
A number of papers devoted to these problems are seen in the references in \cite{J-M-W}.
Stability problems in the pathwise sense for solutions of Brownian SDEs with non-Lipschitz coefficients have been discussed in Kawabata-Yamada \cite{Kw-Y}.

In the theory of SDEs driven by a symmetric $\alpha$ stable process, some non-Lipschitz conditions for coefficients which guarantee the pathwise uniqueness for solutions are known (\cite{Bass}, \cite{B-O}, \cite{H-T-Y}, \cite{Komatsu}, \cite{Tsuchiya}, \cite{Zanzotto}).
Komatsu condition (\cite{Komatsu}, see also \cite{Bass}) is an analogue of Yamada-Watanabe condition for one-dimensional Brownian SDEs. 
A Nagumo type modification of Komatsu condition is shown in \cite{H-T-Y}.
In multi-dimensional case, Tsuchiya \cite{Tsuchiya} considered rather recently the pathwise uniqueness of solutions of SDEs driven by a symmetric $\alpha$ process.
Belfadli-Ouknine condition which was found very recently can be seen as the counterpart of Nakao-Le Gall condition in the Brownian motion case (\cite{B-O}, \cite{LeGall}, \cite{Nakao}).

In this situation, it seems to be very natural to investigate Euler-Maruyama approximation as well as stability problems under some non-Lipschitz conditions in the case where SDEs are driven by a symmetric $\alpha$ stable process.

Our paper is organized as follows.

In section \ref{sec:2}
, Euler-Maruyama approximation in the pathwise sense is shown under Komatsu condition for coefficients.
Theorem \ref{th:1} in the section corresponds to the main result stated in \cite{Yamada} for Brownian SDEs.
Euler-Maruyama approximation in this section implies naturally the existence of strong solutions for SDEs driven by a symmetric $\alpha$ stable process.

In section \ref{sec:3}, the stability of solutions for SDEs in the pathwise sense under Komatsu condition is proved.
The related result in Brownian motion case has been given in \cite{Kw-Y}.

In section \ref{sec:4}, the stability of solutions also in the pathwise sense is discussed under Belfadli-Ouknine condition for coefficients.
\section{Euler-Maruyama approximation}\label{sec:2}
Let $(\Omega,{\cal F}, \{ {\cal F} _t \},{\bf P})$ be a filtered probability space with usual conditions and $Z=\{ Z(t) ; t \geq 0 \}$ be a ${\cal F}_t$-symmetric $\alpha$ stable process such that $Z(0) = 0$, $c\grave{a}dl\grave{a}g$ (right continuous left limit) ${\cal F}_t$-adapted and
\begin{eqnarray*}
{\bf E}[\exp \{ i\xi (Z(t) - Z(s)) \} | {\cal F}_s] = \exp \{ -(t-s)|\xi|^\alpha\} \ {\rm a.s. \ for \ any \ } s < t,\ \xi \in \mathbb{R} .
\end{eqnarray*}

In the present section we consider the following SDE:
\begin{eqnarray}
X(t) = X(0) + \int_0^t \sigma(s,X(s-)) dZ(s) , \label{eq:1}
\end{eqnarray}
where $Z(t)$ is a symmetric $\alpha$ stable process ($1 < \alpha < 2$), and $\sigma(t,x)$ is a Borel measurable real function with respect to $(t,x)$.

We assume that the coefficient function $\sigma(t,x)$ satisfies the following condition:\\

\noindent
Condition (A)
\begin{enumerate}[(i)]
\item
there exists a positive constant $M_1$ such that $|\sigma(t,x)| \leq M_1$,
\item
$\sigma(t,x)$ is uniformly continuous on $[0,\infty) \times \mathbb{R}$,
\item
there exists a non-negative increasing function $\rho$ defined on $[0,\infty)$ such that: $\rho(0) = 0$, $\int_{0+} \rho^{-1}(x) dx =\infty$
\begin{eqnarray*}
|\sigma(t,x)-\sigma(t,y)|^\alpha \leq \rho(|x-y|),
 \quad \forall x, \forall y\in \mathbb{R}.
\end{eqnarray*}
\end{enumerate}

\begin{remark}
The condition (A) is called Komatsu condition. Komatsu \cite{Komatsu} has proved that under the condition (A) for given bounded initial value $X(0)$ the pathwise uniqueness holds for (\ref{eq:1}) (see also \cite{Bass}).
\end{remark}

\smallskip
Let $0 < T < \infty$ be a fixed constant.
Let $\Delta$ be a partition of the interval $[0 , T]$, such that $ \Delta : 0 = t_0 < t_1 < \dots < t_k < t_{k+1} < \dots < t_n =T$ .
The norm of $\Delta$, $\| \Delta \|$ is defined as 
$\| \Delta \| := \max_{1 \leq k \leq n} (t_k - t_{k-1})$, and we put $\eta_\Delta(t):=t_k $ for $ t_k \leq t < t_{k+1}.$

Euler-Maruyama approximation for (\ref{eq:1}) is the following:
\begin{align*}
X_\Delta(0) &:= X(0),\\
{\rm and} \\
X_\Delta(t_k) &:= X_\Delta(t_{k-1}) + \sigma(t_{k-1},X_\Delta(t_{k-1}-))(Z(t_k)-Z(t_{k-1})), \quad k=1,2, \dots ,n.
\end{align*}
For $t_k \leq t < t_{k+1}$, $X_\Delta(t)$ is defined as
\begin{eqnarray*}
X_\Delta(t) := X_\Delta(t_k) + \sigma(t_k,X_\Delta(t_k-))(Z(t)-Z(t_{k})).
\end{eqnarray*}
Using the notation $\eta_\Delta$, $X_\Delta(t)$ satisfies the equation:
\begin{eqnarray*}
X_\Delta(t) := X(0) + \int_0^t\sigma(\eta_\Delta(s),X_\Delta(\eta_\Delta(s)-))dZ(s).
\end{eqnarray*}
\begin{theorem}\label{th:1}
We assume condition (A) for (\ref{eq:1}).
Let $X(t)$ be a unique solution of (\ref{eq:1}) with bounded initial value $X(0)$.
Then Euler-Maruyama approximation $X_\Delta(t)$ satisfies
\begin{eqnarray*}
\lim_{\|\Delta \| \to 0} {\bf E} \Bigl[ \sup_{0 \leq t \leq T}|X_\Delta(t) - X(t)|^\beta \Bigr] = 0 \quad for \ any \ \beta \in (1,\alpha).
\end{eqnarray*}
\end{theorem}

For the proof of Theorem \ref{th:1}, we prepare several lemmas.

\begin{lemma}\label{lem:1}
Under the same assumption as in Theorem \ref{th:1}, we have
\begin{eqnarray*}
\lim_{|u-v| \to 0} {\bf E} [|X_\Delta(u) - X_\Delta(v)|^\beta] = 0 \quad for \ \beta \in (1,\alpha),\ uniformly \ with \ respect \ to \ \Delta.
\end{eqnarray*}
\end{lemma}

\begin{proof}
\smartqed
Choose $\beta'$ and $p > 0$ such that $1 < \beta < \beta' <\alpha$, and $1/p + 1/{\beta'} = 1/{\beta}$.
For $v \leq u$, we see
\begin{eqnarray*}
X_\Delta(u) - X_\Delta(v) = \int_v^u \sigma(\eta_\Delta(s),X_\Delta(\eta_\Delta(s)-))dZ(s).
\end{eqnarray*}
Let $[Y,Y]$ be the quadratic variation of a semimartingale $Y$ (see for examples \cite{D-M},\cite{Protter}).
By \'{E}mery's inequality (\cite{Emery}, page 191 in \cite{Protter}),
\begin{eqnarray}
\lefteqn{{\bf E} [([X_\Delta,X_\Delta](u) - [X_\Delta,X_\Delta](v))^{\beta/2} ]^{1/\beta} } \hspace{1cm} \nonumber \\
&\leq& {\bf E} \Bigl[\sup_{v \leq s \leq u}|\sigma(\eta_\Delta(s),X_\Delta(\eta_\Delta(s)-))|^p\Bigr]^{1/p} {\bf E} [([Z,Z](u) -[Z,Z](v))^{\beta'/2}]^{1/{\beta'}}. \nonumber \\
& & {} \label{eq:2}
\end{eqnarray}
Also, by Burkholder-Davis-Gundy's inequality (see for examples \cite{D-M},\cite{Protter}), we have
\begin{eqnarray}
\lefteqn{ c_\beta {\bf E} \Bigl[\sup_{v \leq s \leq u} |X_\Delta(s) - X_\Delta(v)|^\beta \Bigr]^{1/\beta} } \hspace{2cm} \nonumber \\
&\leq& {\bf E} [([X_\Delta,X_\Delta](u) - [X_\Delta,X_\Delta](v))^{\beta/2} ]^{1/\beta} \label{eq:3}
\end{eqnarray}
and
\begin{eqnarray}
\lefteqn{ {\bf E} [([Z,Z](u) -[Z,Z](v))^{\beta'/2}]^{1/{\beta'}} } \hspace{1cm} \nonumber \\
&\leq& C_{\beta'} {\bf E} \Bigl[\sup_{v \leq s \leq u} |Z(s) - Z(v)|^{\beta'}\Bigr]^{1/{\beta'}} \label{eq:4}
\end{eqnarray}
where $c_\beta$ and $C_{\beta'}$ are positive constants which depend on $\beta$ and $\beta'$ with respectively.
By (i) in (A), the right-hand side of (\ref{eq:2}) can be bounded by $M_1{\bf E} [([Z,Z](u) - [Z,Z](v))^{{\beta'}/2}]^{1/{\beta'}}$.
Then by (\ref{eq:2}), (\ref{eq:3}) and (\ref{eq:4}), we have
\begin{eqnarray*}
{\bf E} \Bigl[ \sup_{v \leq s \leq u} |X_\Delta(s) - X_\Delta(v)|^\beta \Bigr]^{1/\beta}
\leq \frac{C_{\beta'}}{c_\beta} M_1 {\bf E} \Bigl[ \sup_{v \leq s \leq u}|Z(s) -Z(v)|^{\beta'}\Bigr]^{1/{\beta'}}.
\end{eqnarray*}
Using Doob's inequality (see for example \cite{R-Y})
\begin{eqnarray*}
{\bf E} \Bigl[ \sup_{0 \leq s \leq u-v} |Z(s)|^{\beta'} \Bigr]^{1/{\beta'}}
\leq \frac{\beta'}{\beta'-1} {\bf E} [|Z(u-v)|^{\beta'}]^{1/{\beta'}},
\end{eqnarray*}
we have
\begin{eqnarray*}
{\bf E} \Bigl[ \sup_{v \leq s \leq u} |X_\Delta(s) - X_\Delta(v)|^\beta \Bigr]^{1/\beta}
\leq \frac{C_{\beta'}}{c_\beta} \frac{\beta'}{\beta'-1} M_1 {\bf E}[|Z(u-v)|^{\beta'}]^{1/{\beta'}}.
\end{eqnarray*}
The right-hand side in the above inequality does not depend on $\Delta$.
Thus, we can conclude
\begin{eqnarray*}
\lim_{|u - v| \to 0} \|X_\Delta(u) - X_\Delta(v) \|_{L^\beta} = 0,\quad {\rm uniformly \ with \ respect \ to \ \Delta}.
\end{eqnarray*}
\qed
\end{proof}

\begin{lemma}\label{lem:2}
Under the same assumption as in Theorem \ref{th:1}, we have
\begin{eqnarray*}
\lim_{ \| \Delta \| \to 0} {\bf E} [|X_\Delta(t) - X(t)|^{\alpha-1} ] = 0,
\quad for \ t \in [0,T].
\end{eqnarray*}
\end{lemma}

\begin{proof}
\smartqed

As is well known (for example see \cite{Komatsu}) that the generator ${\cal L}$ of $Z(t)$ is defined by
\begin{eqnarray*}
{\cal L}f(x) := \int[f(x+y)-f(x)-I_{\{ |y| \leq 1\}}yf'(x)]|y|^{-1-\alpha} dy.
\end{eqnarray*}
We choose a sequence $ \{ a_m \} $ such that $ 1=a_0 > a_1 > \cdots $ and \\
$ \int_{a_m}^{a_{m-1}} \rho^{-1}(x)dx = m$. 
For this choice we can choose a sequence of sufficiently smooth even functions $ \{ \varphi_m \}$ such that,
\begin{eqnarray}
\varphi_m(x)=\begin{cases}
          0  & |x| \leq a_m \\
          {\rm between} \ 0 \ {\rm and} \ 1/(m\rho(x))  & a_m < |x| < a_{m-1} \\
          0  & |x| \geq a_{m-1}  \label{eq:5}
          \end{cases}
\end{eqnarray}
and $\int_{-\infty}^\infty \varphi_m(x) dx = 1$. 
Following the argument employed by Komatsu \cite{Komatsu}, if we put $u(x) := |x|^{\alpha-1}$ and also $u_m := u \ast \varphi_m$ where $\ast$ stands for the convolution operetor, then we have ${\cal L}u_m=K_\alpha \varphi_m$, where $K_\alpha=-2 \pi \alpha^{-1} \cot(\alpha \pi / 2)$. 
We note that $K_\alpha$ does not depend on $m$.

By the definition of $u_m$, we have
\begin{eqnarray}
|x|^{\alpha-1} - a_{m-1}^{\alpha-1} 
 \leq u_m(x) 
 \leq |x|^{\alpha-1} + a_{m-1}^{\alpha-1}. \label{eq:6}
\end{eqnarray}

On the other hand , It\^{o} formula implies
\begin{eqnarray*}
\lefteqn{u_m(X_\Delta(t) - X(t))} \hspace{1cm} \\
&=& K_\alpha \int_0^t \varphi_m(X_\Delta(s) - X(s)) |\sigma(\eta_\Delta(s),X_\Delta(\eta_\Delta(s)-)) - \sigma(s,X(s-))|^\alpha ds \\
& & {} + M_m(t) 
\end{eqnarray*}
where $M_m(t)$ is a martingale.
So, we have
\begin{eqnarray*}
\lefteqn{{\bf E} [u_m(X_\Delta(t) - X(t))]} \hspace{1cm} \\
&=& K_\alpha {\bf E} \Bigl[ \int_0^t \varphi_m(X_\Delta(s) - X(s)) \\
& & \qquad |\sigma(\eta_\Delta(s),X_\Delta(\eta_\Delta(s)-))
        - \sigma(s,X(s-))|^\alpha ds \Bigr].
\end{eqnarray*}

Using the left-hand side inequality of (\ref{eq:6}), we have
\begin{eqnarray*}
0 
& \leq & {\bf E} [|X_\Delta(t) - X(t)|^{\alpha-1}] \\
& \leq & a_{m-1}^{\alpha-1} 
    + K_\alpha {\bf E} \Bigl[ \int_0^t \varphi_m(X_\Delta(s) - X(s)) \\
& & \qquad |\sigma(\eta_\Delta(s),X_\Delta(\eta_\Delta(s)-))
      - \sigma(s,X(s-))|^\alpha ds \Bigr] \\
& \leq & a_{m-1}^{\alpha-1} 
    + 2K_\alpha {\bf E} \Bigl[\int_0^t \| \varphi_m \| 
   |\sigma(\eta_\Delta(s),X_\Delta(\eta_\Delta(s)-))
    - \sigma(s,X_\Delta(s))|^\alpha ds \Bigr] \\
& &{} + 2K_\alpha {\bf E} \Bigl[ \int_0^t \varphi_m(X_\Delta(s) - X(s)) 
   |\sigma(s,X_\Delta(s))
    - \sigma(s,X(s-))|^\alpha ds \Bigr] \\
&=& a_{m-1}^{\alpha-1} + J_\Delta^{(1)} + J_\Delta^{(2)} \quad {\rm say.}
\end{eqnarray*}
By (iii) in (A), (\ref{eq:5}) implies $J_\Delta^{(2)} \leq (2 K_\alpha T)/m$.
Let $\varepsilon > 0$ be fixed, we choose a fixed integer $m$ such that $a_{m-1}^{\alpha-1} < \varepsilon / 3$ and $J_\Delta^{(2)} < \varepsilon / 3$.
By (ii) in (A), we choose $\delta_1 > 0$ such that, for $|t - t'| < \delta_1, \ |x - x'| < \delta_1$
\begin{eqnarray}
|\sigma(t,x) - \sigma(t',x')| < \frac{\varepsilon}{12 K_\alpha T \|\varphi_m\|} \label{eq:7}
\end{eqnarray}
holds.
Lemma \ref{lem:1} implies immediately 
\begin{eqnarray*}
\lim_{|u-v| \to 0} |X_\Delta(u) - X_\Delta(v)| = 0 \quad  {\rm (in \ probability),\ uniformly \ with \ respect \ to \ \Delta }.
\end{eqnarray*}
So we can choose $\delta_2 > 0$ such that for $|u - v| < \delta_2$
\begin{eqnarray}
{\bf P}(|X_\Delta(u) - X_\Delta(v)| > \delta_1) \leq \frac{\varepsilon}{12 K_\alpha T \| \varphi_m \|(2M_1)^\alpha} \label{eq:8}
\end{eqnarray}
holds.

Let $\|\Delta\| \leq \min(\delta_1,\delta_2)$. Then we have $|t-\eta_\Delta(t)| \leq \min(\delta_1,\delta_2)$. The inequalities (\ref{eq:7}) and (\ref{eq:8}) imply
\begin{eqnarray*}
J_\Delta^{(1)}
&=& 2 K_\alpha \| \varphi_m \| 
    {\bf E} \Bigl[ I_{\{ | X_\Delta(s) - X_\Delta(\eta_\Delta(s)) | \leq \delta_1\} } \\
& & \quad {} \int_0^t |\sigma(s,X(s)) - \sigma(\eta_\Delta(s),X_\Delta(\eta_\Delta(s)-))|^\alpha ds \Bigr] \\
& & + 2 K_\alpha \| \varphi_m \| 
    {\bf E} \Bigl[ I_{\{ | X_\Delta(s) - X_\Delta(\eta_\Delta(s)) | > \delta_1\} } \\
& & \quad {} \int_0^t |\sigma(s,X(s)) - \sigma(\eta_\Delta(s),X_\Delta(\eta_\Delta(s)-))|^\alpha ds \Bigr] \\
&<& 2 K_\alpha \| \varphi_m \| T\frac{\varepsilon}{12 K_\alpha T \| \varphi_m \|}
   + 2 K_\alpha \| \varphi_m \| T\frac{\varepsilon (2M_1)^\alpha}{12 K_\alpha T \| \varphi_m \|(2M_1)^\alpha}\\
&=& \frac{\varepsilon}{3}.
\end{eqnarray*}
Thus, we can conclude that 
\begin{eqnarray*}
\lim_{ \| \Delta \| \to 0 } {\bf E} [|X_\Delta(t) - X(t)|^{\alpha-1} ] = 0.
\end{eqnarray*}
\qed
\end{proof}

\begin{lemma}\label{lem:3}
Under the same assumption as in Theorem \ref{th:1},
\begin{enumerate}[{\rm {[}a{]}}]
\item
$ \sup_{0 \leq t \leq T}|X_\Delta(t) - X(t)| \to 0$ in probability $ ( \| \Delta \| \to 0 )$,
\item
the class of random variables
\begin{eqnarray*}
\Bigl\{ \sup_{0 \leq t \leq T} | X_\Delta(t) - X(t) | ^ \beta, \Delta \Bigr\}
\end{eqnarray*}
is uniformly integrable.
\end{enumerate}
\end{lemma}

\begin{proof}
\smartqed

We consider the probability ${\bf P}( \sup_{0 \leq t \leq T} | X_\Delta(t) - X(t)| > \lambda)$.
By Gin\'{e}-Marcus's inequality (page 213 in \cite{Applebaum}, \cite{G-M}) there exists a constant $C > 0$ independent of $\lambda > 0$, such that
\begin{eqnarray}
\lefteqn{ {\bf P} \Bigl( \sup_{0 \leq t \leq T} |X_\Delta(t) - X(t) | > \lambda \Bigr) \nonumber} \\
&=& {\bf P} \Bigl( \sup_{0 \leq t \leq T} \Bigl| \int_0^t \{\sigma(\eta_\Delta(s),X_\Delta(\eta_\Delta(s)-)) - \sigma(s,X(s-))\}dZ(s) \Bigr| > \lambda \Bigr) \nonumber \\
&\leq& \lambda^{-\alpha} C \int_0^T {\bf E} [|\sigma(\eta_\Delta(s),X_\Delta(\eta_\Delta(s)-)) - \sigma(s,X(s-))|^\alpha]ds. \nonumber \\
& & {} \label{eq:9}
\end{eqnarray}
Noticing that
\begin{eqnarray*}
\lefteqn{ |X_\Delta(\eta_\Delta(s)-) - X(s-)| } \hspace{1cm} \\
&\leq& |X_\Delta(\eta_\Delta(s)-) - X_\Delta(s-)| + |X_\Delta(s-) - X(s-)|,
\end{eqnarray*}
and also note that for fixed $u \in [0,T]$, $X_\Delta(u) = X_\Delta(u-)$ a.s. hold.
By Lemma \ref{lem:1} and Lemma \ref{lem:2}, we can conclude that $|X_\Delta(\eta_\Delta(s)-) - X(s-)|$ converges to 0 in probability when $\|\Delta\| \to 0$.
Since the function $\sigma(t,x)$ is bounded and uniformly continuous with respect to $(t,x)$, the inequality (\ref{eq:9}) implies [a].

Choose $\beta', \tilde{\beta},p $ such that $\beta < {\beta'} < \tilde{\beta} < \alpha$, $1 / \tilde{\beta} + 1 / p = 1 / \beta'$.
By an analogous argument as in the proof of Lemma \ref{lem:1}, we have
\begin{eqnarray*}
\lefteqn{ {\bf E} \Bigl[ \sup_{0 \leq t \leq T} |X_\Delta(t) - X(t)|^{\beta'} \Bigr]^{1/{\beta'}} } \hspace{1cm} \\
&\leq& \frac{1}{c_{\beta'}} {\bf E} [([X_\Delta - X , X_\Delta - X](T))^{{\beta'}/2}]^{1/{\beta'}} \\
&\leq& \frac{1}{c_{\beta'}} {\bf E} \Bigl[\sup_{0 \leq t \leq T} |\sigma(\eta_\Delta(s),X_\Delta(\eta_\Delta(s)-)) - \sigma(s,X(s-))|^p \Bigr]^{1/p} {\bf E}[([Z,Z](T))^{\tilde{\beta}/2}]^{1/{\tilde{\beta}}} \\
&\leq& \frac{2M_1}{c_{\beta'}} {\bf E} [([Z,Z](T))^{{\tilde{\beta}}/2}]^{1/{\tilde{\beta}}} \\
&\leq& 2M_1 \frac{C_{\tilde{\beta}}}{{c_{\beta'}}} {\bf E} \Bigl[ \sup_{0 \leq t \leq T} |Z(t)|^{\tilde{\beta}}\Bigr]^{1/\tilde{\beta}} \\
&\leq& 2M_1 \frac{C_{\tilde{\beta}}}{{c_{\beta'}}} \frac{\tilde{\beta}}{\tilde{\beta}-1} {\bf E} [|Z(T)|^{\tilde{\beta}}]^{1/\tilde{\beta}} \\
&<& \infty .
\end{eqnarray*}
Thus, we can conclude that the class $ \{ \sup_{0 \leq t \leq T} |X_\Delta(t) - X(t)|^\beta, \Delta \}$ is uniformly integrable.
\qed
\end{proof}

\noindent
\textit{Proof of Theorem \ref{th:1}.}
\smartqed
Lemma \ref{lem:3} implies immediately Theorem \ref{th:1}.
\qed

\begin{remark}[A consturuction of the strong solution of (\ref{eq:1})]
Let $X_\Delta(t)$ and $X_{\Delta '}(t)$ be two Euler-Maruyama approximations with the same bounded intial value $X(0)$.
By an analogous method emplyed in the proof of Theorem \ref{th:1}, we can show
\begin{eqnarray}
\lim_{\| \Delta \| \to 0, \| \Delta ' \| \to 0,} {\bf E} \Bigl[ \sup_{0 \leq t \leq T} |X_\Delta(t) - X_{\Delta '}(t)|^\beta \Bigr] = 0, \qquad \forall \beta < \alpha \label{eq:10}
\end{eqnarray}
without using the existance of a solution of (\ref{eq:1}).
From this fact we can construct a strong solution of (\ref{eq:1}) with initial value $X(0)$, in the following way.

Choose a sequence of positive numbers $\varepsilon_i > 0, i=1,2,\cdots$ such that
\begin{eqnarray}
\sum_{i=1}^\infty 4^i \varepsilon_i < \infty . \label{eq:11}
\end{eqnarray}
By (\ref{eq:10}) we can choose a series of partitions $\Delta_i, i=1,2,\cdots$ such that
\begin{enumerate}[(i)]
\item
$ \| \Delta_i \| \to 0$ ($i \to \infty$),
\item
$ {\bf E}[\sup_{0 \leq t \leq T} |X_{\Delta_i}(t) - X_{\Delta_{i+1}}(t)|^\beta] < \varepsilon_i$, $i=1,2,\cdots.$
\end{enumerate}
Since
\begin{eqnarray*}
\lefteqn{ {\bf P} \Bigl( \sup_{0 \leq t \leq T} |X_{\Delta_i}(t) - X_{\Delta_{i+1}}(t)| > \frac{1}{2^i} \Bigr) \hspace{1cm}} \\
&=& {\bf P} \Bigl( \sup_{0 \leq t \leq T} |X_{\Delta_i}(t) - X_{\Delta_{i+1}}(t)|^\beta > \bigl( \frac{1}{2^i} \bigr)^\beta \Bigr) \\
&\leq& {\bf P} \Bigl( \sup_{0 \leq t \leq T} |X_{\Delta_i}(t) - X_{\Delta_{i+1}}(t)|^\beta > \frac{1}{4^i} \Bigr) \\
&\leq& 4^i {\bf E} \Bigl[ \sup_{0 \leq t \leq T} |X_{\Delta_i}(t) - X_{\Delta_{i+1}}(t)|^\beta \Bigr] \\
&<& 4^i \varepsilon_i,
\end{eqnarray*}
we have by (\ref{eq:11}),
\begin{eqnarray*}
\sum_{i=1}^\infty {\bf P} \Bigl( \sup_{0 \leq t \leq T} |X_{\Delta_i}(t) - X_{\Delta_{i+1}}(t)| > \frac{1}{2^i} \Bigr) 
< \sum_{i=1}^\infty 4^i \varepsilon_i 
< \infty.
\end{eqnarray*}
Then, by Borel-Cantelli lemma $X_{\Delta_i}(t)$ converges uniformly on $[0,T]$ a.s. $(i \to \infty)$.
Put $X(t) := \lim_{i \to \infty} X_{\Delta_i}(t)$, $t \in [0,T]$.
Then, we see that $X(t)$ is a $c\grave{a}dl\grave{a}g$, and we have
\begin{eqnarray*}
\lim_{i \to \infty} {\bf E} \Bigl[\sup_{0 \leq t \leq T} |X_{\Delta_i}(t) - X(t)|^\beta \Bigr] = 0.
\end{eqnarray*}
We will show that $(X(t), Z(t))$ satisfies (\ref{eq:1}).
For the purpose of the proof, we have
\begin{eqnarray*}
& &{\bf E} \Bigl[\sup_{0 \leq t \leq T} \Bigl| X(t) - X(0) - \int_0^t \sigma(s,X(s-))dZ(s) \Bigr|^\beta \Bigr] \\
&\leq& 2{\bf E} \Bigl[ \sup_{0 \leq t \leq T} |X(t) - X_{\Delta_i}(t)|^\beta \Bigr] \\
& & {} + 2{\bf E} \Bigl[ \sup_{0 \leq t \leq T} \Bigl| \int_0^t \{\sigma(s,X(s-)) - \sigma(\eta_{\Delta_i}(s),X_{\Delta_i}(\eta_{\Delta_i}(s)-))\}dZ(s) \Bigr|^\beta \Bigr] \\
&=& {\bf E} [N_i^{(1)}] + {\bf E} [N_i^{(2)}] \quad {\rm say.}
\end{eqnarray*}
Obviously, we have $ \lim_{i \to \infty} {\bf E}[N_i^{(1)}] = 0$.
By the same discussion employed in the proof of Lemma \ref{lem:3}, we can see that $ \lim_{i \to \infty} {\bf E}[N_i^{(2)}] = 0$.
Thus, we have
\begin{eqnarray*}
{\bf E} \Bigl[ \sup_{0 \leq t \leq T} \Bigl| X(t) - X(0) - \int_0^t \sigma(s,X(s-))dZ(s) \Bigr|^\beta \Bigr] = 0.
\end{eqnarray*}
So, we can conclude
\begin{eqnarray*}
X(t) = X(0) + \int_0^t \sigma(s,X(s-))dZ(s).
\end{eqnarray*}
By the definition of Euler-Maruyama approximation, $X_{\Delta_i}(t)$ is $\sigma(Z(s); 0 \leq s \leq t)$ measurable.
This implies immediately $X(t)$ is $\sigma(Z(s); 0 \leq s \leq t)$ measurable.
It means that $X(t)$ is a strong soluiton of (\ref{eq:1}).
\end{remark}
\section{Stability of solutions under Komatsu condition}\label{sec:3}
Consider the following sequence of SDEs driven by a same symmetric $\alpha$ stable process $Z(t)$:
\begin{eqnarray}
X(t) &=& X(0) + \int_0^t \sigma(s,X(s-)) dZ(s).  \label{eq:12} \\
X_n(t) &=& X_n(0) + \int_0^t \sigma_n(s,X_n(s-)) dZ(s), \quad {\rm n=1,2,\cdots} \label{eq:13}
\end{eqnarray}
We assume that the coefficient functions $\sigma(t,x)$, $\sigma_n(t,x)$, $n=1,2,\cdots$ satisfy the following condition. \\

\noindent
Condition (B)
\begin{enumerate}[(i)]
\item
there exists a positive constant $M_2$ such that $|\sigma(t,x)| \leq M_2$, $|\sigma_n(t,x)| \leq M_2$, $n=1,2,\cdots$,
\item
$\lim_{n \to \infty} \sup_{t,x} |\sigma_n(t,x) - \sigma(t,x)| = 0$,
\item
there exists an increasing function $\rho$ defined on $[0,\infty)$ such that $\rho(0) = 0$, $\int_{0+} \rho^{-1}(x) dx = \infty$
\begin{eqnarray*}
|\sigma_n(t,x)-\sigma_n(t,y)|^\alpha &\leq& \rho(|x-y|),
 \quad \forall x, \forall y \in \mathbb{R}, \ t \in [0,\infty), \ n=1,2,\cdots \\
|\sigma(t,x)-\sigma(t,y)|^\alpha &\leq& \rho(|x-y|),
 \quad \forall x, \forall y \in \mathbb{R}, \ t \in [0,\infty) \\
\end{eqnarray*}
\end{enumerate}

\begin{remark}
The pathwise uniqueness holds for solutions of (\ref{eq:12}) and (\ref{eq:13}) under the condition (B) (\cite{Bass}, \cite{Komatsu}).
\end{remark}

The main result of this section is following:

\begin{theorem}\label{th:2}
Let $T > 0$ be fixed. 
Assume that there exists a positive number $M_0$ such that $|X(0)| \leq M_0$, a.s. and $|X_n(0)| \leq M_0$, a.s., $n=1,2,\cdots$ .
Assume also that
\begin{eqnarray*}
\lim_{n \to \infty} {\bf E} [ |X_n(0) - X(0)|^\alpha ] = 0 .
\end{eqnarray*}
Then under the condition (B)
\begin{eqnarray*}
\lim_{n \to \infty} {\bf E} \Bigl[ \sup_{0 \leq t \leq T}|X_n(t) - X(t)|^\beta \Bigr] = 0 
\end{eqnarray*}
holds, for $\beta < \alpha$.
\end{theorem}

For the proof of Theorem \ref{th:2}, we prepare folloing lemmas.
\begin{lemma}\label{lem:4}
Under the same assumption as in Theorem \ref{th:2},
\begin{eqnarray*}
\lim_{n \to \infty} {\bf E} [|X_n(t) - X(t)|^{\alpha-1} ] = 0 \ holds, \ t \in [0,T].
\end{eqnarray*}
\end{lemma}

\begin{proof}
\smartqed
Put $u(x):=|x|^{\alpha-1}$.
Choose a sequence $ \{ a_m \} $ such that $ 1=a_0 > a_1 > \cdots $, $ \lim \limits_{m \to \infty} a_m =0 $ and $ \int_{a_m}^{a_{m-1}} \rho^{-1}(x)dx = m$.
For this choice, choose a sequence $ \{ \varphi_m \}, \ m=1,2,\cdots$ of sufficiently smooth even functions such that
\begin{eqnarray}
\varphi_m(x)=\begin{cases}
          0  & |x| \leq a_m \\
          {\rm between} \ 0 \ {\rm and} \ 1/(m\rho(x))  & a_m < |x| < a_{m-1} \label{eq:14}\\
          0  & |x| \geq a_{m-1}
          \end{cases}
\end{eqnarray}
and $\int_{-\infty}^\infty \varphi_m(x) dx = 1$. 
After Komatsu, put $u_m:=u \ast \varphi_m$, then we have ${\cal L} u_m = K_\alpha \varphi_m$, where $K_\alpha = -2 \pi \alpha^{-1} \cot(\alpha \pi/2)$.
By (ii) of (B) we can choose a sequence $\{ \varepsilon_n \} $ of decreasing positive numbers $\varepsilon_n \downarrow 0$ such that
\begin{eqnarray}
\sup_{t,x} |\sigma_n(t,x) - \sigma(t,x)|^\alpha \leq \varepsilon_n  \label{eq:15}.
\end{eqnarray}
Corresponding this choice, we can find a sequence of integer numbers $\{ m_n\}, \ n=1,2,\cdots$, $m_n \to \infty \ (n \to \infty)$ such that
\begin{eqnarray}
\varepsilon_n \max_{a_{m_n} \leq x \leq a_{m_n-1}} \frac{1}{\rho(x)} \leq 1 \label{eq:16}.
\end{eqnarray}
By It\^{o} formula,
\begin{eqnarray*}
\lefteqn{u_{m_n}(X_n(t) - X(t)) - u_{m_n}(X_n(0) - X(0))} \\
&=& \int_0^t |\sigma_n(s,X_n(s-)) - \sigma(s,X(s-))|^\alpha K_\alpha \varphi_{m_n}(X_n(s) - X(s))ds \\
& & {} + M_{m_n}(t) ,
\end{eqnarray*}
where $M_{m_n}(t)$ is a martingale.
By (\ref{eq:6})
\begin{eqnarray*}
& & {\bf E} [u_{m_n}(X_n(t) - X(t))] \\
&\leq& {\bf E} [|X_n(0) - X(0)|^{\alpha - 1}] +a_{m_n-1}^{\alpha - 1} \\
& & {} + 2 {\bf E} \Bigl[ \int_0^t |\sigma_n(s,X_n(s-)) - \sigma(s,X_n(s-))|^\alpha K_\alpha \varphi_{m_n}(X_n(s) - X(s))ds\Bigr] \\ 
& & {} + 2 {\bf E} \Bigl[ \int_0^t |\sigma(s,X_n(s-)) - \sigma(s,X(s-))|^\alpha K_\alpha \varphi_{m_n}(X_n(s) - X(s))ds \Bigr] \\
&=& {\bf E} [|X_n(0) - X(0)|^{\alpha - 1}] +a_{m_n-1}^{\alpha - 1} + N_{m_n}^{(1)} + N_{m_n}^{(2)} \qquad {\rm say.}
\end{eqnarray*}
By the assumption, $\lim_{n \to \infty} {\bf E} [|X_n(0) - X(0)|^{\alpha - 1}] = 0$ holds.
Obviously $\lim_{n \to \infty} a_{m_n - 1}^{\alpha - 1} = 0$.

For $N_{m_n}^{(1)}$, by (\ref{eq:14}), (\ref{eq:15}) and (\ref{eq:16}) we have
\begin{eqnarray*}
N_{m_n}^{(1)} \leq 2 {\bf E} \Bigl[\int_0^t \varepsilon_n K_\alpha \frac{I_{\{a_{m_n} < |X_n(s) - X(s)| <a_{m_n - 1}\}}}{m_n \rho(|X_n(s) - X(s)|)}ds \Bigr] \leq \frac{2 K_\alpha T}{m_n}.
\end{eqnarray*}
From this $\lim_{n \to \infty} N_{m_n}^{(1)} = 0 $ follows immediately.

For $N_{m_n}^{(2)}$, by (iii) of (B) and by the definition of $\varphi_m$, (\ref{eq:14}) we have
\begin{eqnarray*}
N_{m_n}^{(2)} \leq 2 {\bf E} \Bigl[ \int_0^t \rho(|X_n(s) - X(s)|) K_\alpha \frac{1}{m_n \rho(|X_n(s) - X(s)|)}ds \Bigr] \leq \frac{2 K_\alpha T}{m_n}.
\end{eqnarray*}
From this $\lim_{n \to \infty} N_{m_n}^{(2)} = 0 $ holds.

Note that $\lim_{n \to \infty} u_{m_n}(x) = u(x) = |x|^{\alpha - 1}$, then we have $\lim_{n \to \infty} {\bf E}[|X_n(t) - X(t)|^{\alpha - 1}] = 0$. 
\qed
\end{proof}

\begin{lemma}\label{lem:5}
Under the same assumption as in Theorem \ref{th:2},
\begin{enumerate}[{\rm {[}a{]}}]
\item
$\sup_{0 \leq t \leq T}|X_n(t) - X(t)| \to 0$ in probability ($n \to \infty $),
\item
the family of random variables
\begin{eqnarray*}
\Bigl\{ \sup_{0 \leq t \leq T} | X_n(t) - X(t) |^\beta , \ n = 1, 2, \cdots \Bigr\}
\end{eqnarray*}
is uniformly integrable.
\end{enumerate}
\end{lemma}

\begin{proof}
\smartqed

Let $\lambda$ be a positive constant.
Then, by Gin\'{e}-Marcus's inequality, there exists a constant $C > 0$ which does not depend on $\lambda$ such that
\begin{eqnarray*}
\lefteqn{ {\bf P} \Bigl( \sup_{0 \leq t \leq T} |X_n(t) - X(t) | > \lambda \Bigr) } \hspace{1cm}\\
&\leq& {\bf P} \Bigl( |X_n(0) - X(0) | > \frac{\lambda}{2} \Bigr) \\
& & {} + {\bf P} \Bigl( \sup_{0 \leq t \leq T} \Bigl| \int_0^t \{\sigma_n(s,X_n(s-)) - \sigma(s,X(s-))\}dZ(s) \Bigr| > \frac{\lambda}{2} \Bigr) \\
&\leq& {\bf P} \Bigl( |X_n(0) - X(0) | > \frac{\lambda}{2} \Bigr) \\
& & {} + \Bigl( \frac{\lambda}{2} \Bigr)^{-\alpha} C \int_0^T {\bf E} [|\sigma_n(s,X_n(s-)) - \sigma_n(s,X(s-))|^\alpha]ds. \\
& & {} + \Bigl( \frac{\lambda}{2} \Bigr)^{-\alpha} C \int_0^T {\bf E} [|\sigma_n(s,X(s-)) - \sigma(s,X(s-))|^\alpha]ds. \\
&=& J_n^{(1)} + J_n^{(2)} + J_n^{(3)} \quad {\rm say}.
\end{eqnarray*}
By the assumption on initial data, $\lim_{n \to \infty} J_n^{(1)} = 0$ is obvious.
By (ii) of (B), $\lim_{n \to \infty} J_n^{(3)} = 0 $ holds.
By (iii) of (B) we have
\begin{eqnarray*}
 J_n^{(2)} \leq \Bigl( \frac{\lambda}{2} \Bigr)^{-\alpha} C \int_0^T {\bf E} [\rho(|X_n(s-) - X(s-)|)]ds.
\end{eqnarray*}
By the result of Lemma \ref{lem:4}, for fixed $s \in [0,T]$, $X_n(s-) \to X(s-)$ in probability ($ n \to \infty $).
Since $\rho$ is continuous and $\rho(0) = 0$, we have $\lim_{n \to \infty} {\bf E} [\rho(|X_n(s-) - X(s-)|)] = 0 $.
Note that ${\bf E} [\rho(|X_n(s-) - X(s-)|)]$ is uniformly bounded with respect to $s \in [0,T]$.
Then by Lebesgue convergence theorem $\lim_{n \to \infty} J_n^{(2)} = 0 $ holds.%
Thus, we can conclude that [a] follows.

Just the same argument employed in the proof of Lemma \ref{lem:3} implies [b].
\qed
\end{proof}

\noindent
\textit{Proof of Theorem \ref{th:2}.}
\smartqed
Lemma \ref{lem:5} implies immediately Theorem \ref{th:2}.
\qed
\section{Stability of solutions under Belfadli-Ouknine condition}\label{sec:4}
In this section we consider the following sequence of SDEs:
\begin{eqnarray}
X(t) &=& X(0) + \int_0^t \sigma(X(s-)) dZ(s) \label{eq:17} \\
X_n(t) &=& X_n(0) + \int_0^t \sigma_n(X_n(s-)) dZ(s), \quad n=1,2,\cdots \label{eq:18}
\end{eqnarray}
where $Z(t)$ is a symmetric $\alpha$ stable process $(1 < \alpha < 2)$, and coefficient functions $\sigma$, $\sigma_n$ are Borel measurable.

We assume that coefficient functions $\sigma$, $\sigma_n, \ n=1,2,\cdots$ satisfy the following condition. \\

\noindent
Condition (C)
\begin{enumerate}[(i)]
\item
there exists two positive constants $d, K$ such that $0 < d < K < \infty$,
\begin{eqnarray*}
d \leq \sigma(x) &\leq& K \quad \forall x \in \mathbb{R} \\
d \leq \sigma_n(x) &\leq& K \quad \forall x \in \mathbb{R}, \ n=1,2,\cdots
\end{eqnarray*}
\item
there exist an increasing function $f$ such that for every real numbers $x,y$
\begin{eqnarray*}
|\sigma(x)-\sigma(y)|^\alpha &\leq& |f(x) - f(y)| \\
|\sigma_n(x)-\sigma_n(y)|^\alpha &\leq& |f(x) - f(y)|, \quad n=1,2,\cdots
\end{eqnarray*}
\item
\begin{eqnarray*}
\lim_{n \to \infty} \sup_x |\sigma_n(x) - \sigma(x)| = 0.
\end{eqnarray*}
\end{enumerate}

\begin{remark}
Belfadli and Ouknine \cite{B-O} show that under the (ii) of (C) the pathwise uniqueness holds for each solutions of (\ref{eq:17}) and (\ref{eq:18}).
\end{remark}

Let $ T > 0$ be fixed.

\begin{theorem}\label{th:3}
Assume that there exists a positive number $\tilde{M}_0$ such that $|X(0)| \leq \tilde{M}_0$ a.s., $|X_n(0)| \leq \tilde{M}_0$ a.s. $n=1,2,\cdots$ and assume also
\begin{eqnarray*}
\lim_{n \to \infty} {\bf E}[|X_n(0) - X(0)|^\alpha] = 0.
\end{eqnarray*}
Then, under the condition (C)
\begin{eqnarray*}
\lim_{n \to \infty} {\bf E} \Bigl[ \sup_{0 \leq t \leq T} |X_n(t) - X(t)|^\beta \Bigr] = 0
\end{eqnarray*}
holds for $ \beta < \alpha $.
\end{theorem}

For the proof of Theorem \ref{th:3}, we prepare two lemmas.

\begin{lemma}\label{lem:6}
Under the same assumption as in Theorem \ref{th:3},
\begin{eqnarray*}
\lim_{n \to \infty} {\bf E} [|X_n(t) - X(t)|^{\alpha-1} ] = 0,
\quad holds \ for \ t \in [0,T].
\end{eqnarray*}
\end{lemma}

\begin{proof}
\smartqed

Put $u(x):=|x|^{\alpha-1}$.
Choose a sequence $ \{ a_m \} $ such that $ 1=a_0 > a_1 > \cdots, \ a_m \downarrow 0$.
For this choice, we can find a sequence $ \{ \varphi_m \} \ m=1,2,\cdots$ of sufficiently smooth even functions such that $\int_{-\infty}^\infty \varphi_m(x)dx = 1$ and
\begin{eqnarray*}
\varphi_m(x)=\begin{cases}
          0  & |x| \leq a_m \\
          {\rm between} \ 0 \ {\rm and} \ 1/(mx)  & a_m < |x| < a_{m-1} \\
          0  & |x| \geq a_{m-1}
          \end{cases} .
\end{eqnarray*}
Put $u_m := u \ast \varphi_m$, then ${\cal L} u_m = K_\alpha \varphi_m$, where $K_\alpha = -2 \pi \alpha^{-1} \cot(\alpha \pi / 2)$.
Choose also a sequence $\{ \varepsilon_n \} $ of positive numbers such that $\varepsilon_n \downarrow 0$ and
\begin{eqnarray*}
\sup_x |\sigma_n(x) - \sigma(x)| \leq \varepsilon_n .
\end{eqnarray*}
For this choice, we can find a sequence $\{ m_n\}$ of positive integers such that
\begin{eqnarray*}
\varepsilon_n \frac{1}{a_{m_n}} \leq 1 .
\end{eqnarray*}
By It\^{o} formula,
\begin{eqnarray*}
\lefteqn{ {\bf E} [u_{m_n}(X_n(t) - X(t))]  \hspace{1cm}} \\
&=& {\bf E} [u_{m_n}(X_n(0) - X(0))] \\
& & {} + {\bf E} \Bigl[ K_\alpha \int_0^t |\sigma_n(X_n(s-)) - \sigma(X(s-))|^\alpha \varphi_{m_n}(X_n(s-) - X(s-))ds \Bigr] \\
&\leq& {\bf E} [|X_n(0) - X(0)|^{\alpha - 1}] +a_{m_n-1}^{\alpha - 1} \\
& & {} + 2 K_\alpha {\bf E} \Bigl[ \int_0^t |\sigma_n(X_n(s-)) - \sigma_n(X(s-))|^\alpha \varphi_{m_n}(X_n(s-) - X(s-))ds \Bigr] \\
& & {} + 2 K_\alpha {\bf E} \Bigl[ \int_0^t |\sigma_n(X(s-)) - \sigma(X(s-))|^\alpha \varphi_{m_n}(X_n(s-) - X(s-))ds \Bigr] \\
&=& {\bf E} [|X_n(0) - X(0)|^{\alpha - 1}] +a_{m_n-1}^{\alpha - 1} + N_n^{(1)} + N_n^{(2)} \ {\rm say.}
\end{eqnarray*}
Obviously we have $ \lim_{n \to \infty} {\bf E} [|X_n(0) - X(0)|^{\alpha - 1}] = 0$ and $ \lim_{n \to \infty} a_{m_n - 1}^{\alpha - 1} = 0 $.
For $N_n^{(2)}$, we have
\begin{eqnarray*}
N_n^{(2)} &\leq& 2 K_\alpha {\bf E} \Bigl[ \int_0^T \varepsilon_n^\alpha \frac{1}{m_n a_{m_n}}ds \Bigr] \\
&\leq& \frac{2 K_\alpha T}{m_n}.
\end{eqnarray*}
So $\lim_{n \to \infty} N_n^{(2)} = 0 $.

Finally we will discuss $N_n^{(1)}$.
By stopping $X$ and $X_n$, when one of them first leaves a compact set, we can assume $|X| \vee |X_n| \leq M $ for every $t \geq 0$.
Letting $\lambda \downarrow 0$ in the last inequality in the proof of Lemma 2.2 in \cite{B-O}.
We get the following inequality:
\begin{eqnarray*}
N_n^{(1)} \leq 2 K_\alpha \frac{(M+1) \|f\|_\infty}{d^\alpha m_n} \sup_{|a| \leq M+1} \Bigl( \int_0^{K^\alpha T}p_s(a)ds \Bigr)
\end{eqnarray*}
where $\|f\|_\infty = \sup_x |f(x)|$ and $p_s(a) = p_s(a,0)$ is the transition density function of $Z(s)$.
So, we can conclude $\lim_{n \to \infty} N_n^{(1)} = 0$.
The proof of Lemma \ref{lem:6} is achieved.
\qed
\end{proof}

\begin{lemma}\label{lem:7}
Under the same assumption as in Theorem \ref{th:3},
\begin{enumerate}[{\rm {[}a{]}}]
\item
$\sup_{0 \leq t \leq T}|X_n(t) - X(t)| \to 0 $ in probability ($ n \to \infty $),
\item
the family of random variables
\begin{eqnarray*}
\Bigl\{ \sup_{0 \leq t \leq T} | X_n(t) - X(t) |^\beta , \ n = 1, 2, \cdots \Bigr\}
\end{eqnarray*}
is uniformly integrable.
\end{enumerate}
\end{lemma}

The proof of Lemma \ref{lem:7} and Theorem \ref{th:3} follows just the same way as in section \ref{sec:3}.
\begin{acknowledgement}
This work is motivated by fruitful discussions with Toshio Yamada. 
The author would like to thank him very much. 
Professor Shinzo Watanabe reviewed orginal manuscript and offered some polite suggestions.
The author also would like to thank him very much for his valuable comments.
\end{acknowledgement}

\end{document}